\documentclass[12pt, reqno]{amsart}
\usepackage{amsmath, amsthm, amscd, amsfonts, amssymb, graphicx, color}
\usepackage{amssymb,amsmath,amsfonts,latexsym}
\usepackage{bm}
\usepackage[all,cmtip]{xy}
\usepackage{amscd}
\usepackage{fullpage}

\newtheorem{theorem}{Theorem}[section]
\newtheorem{lemma}[theorem]{Lemma}
\newtheorem{proposition}[theorem]{Proposition}
\newtheorem{corollary}[theorem]{Corollary}
\theoremstyle{definition}
\newtheorem{definition}[theorem]{Definition}
\newtheorem{example}[theorem]{Example}

\theoremstyle{remark}
\newtheorem{remark}[theorem]{Remark}
\numberwithin{equation}{section}

\begin{document}



\title[Dynamics]{Hypercyclicity of weighted shifts on weighted Bergman spaces}

\author[Das]{Bibhash Kumar Das}
\address{Indian Institute of Technology Bhubaneswar, Jatni Rd, Khordha - 752050, India}
\email{bkd11@iitbbs.ac.in}

\author[Mundayadan]{Aneesh Mundayadan}
\address{Indian Institute of Technology Bhubaneswar, Jatni Rd, Khordha - 752050, India}
\email{aneesh@iitbbs.ac.in}


\begin{abstract} We study the continuity, and dynamical properties (hypercyclicity, periodic vectors, and chaos) for a weighted backward shift $B_w$ on a weighted Bergman space $A^p_{\phi}$ based on the norm estimates of coefficient functionals on $A^p_{\phi}$. Here, the weight function $\phi(z)$ is mostly radial, but our work will also involve a (non-radial) subharmonic weight. We provide a complete characterization of hypercyclic shifts $B_w$ on $A^p_{\phi}$ when $\phi$ is an (integrable) radial weight. The coefficient multipliers obtained in this paper for certain weights are new.

\end{abstract}
\maketitle

\noindent
\tableofcontents

	\markboth{Bibhash Kumar Das, Aneesh Mundayadan }{Dynamics of weighted shifts}

\section{Introduction}

We begin by recollecting the notions of hypercyclicity, mixing, and chaos from linear dynamics.

 \begin{definition}
 Let $X$ be a separable Banach space. An operator $T:X\rightarrow X$ is said to be
 \begin{itemize}
 \item \textit{hypercyclic} if there exists $x\in X$ whose orbit $\{x,Tx,T^2x,\cdots\}$ is dense in $X$,
\item \textit{topologically transitive} if, for any two non-empty open sets $U_1$ and $U_2$ of $X$, the set $\{k\in \mathbb{N}: T^k(U_1)\cap U_2\neq \varnothing\}$ is non-empty, 
\item \textit{topologically mixing} on $X$ if, for any two non-empty open sets $U_1$ and $U_2$ of $X$, the set $\{k\in \mathbb{N}: T^k(U_1)\cap U_2\neq \varnothing\}$ is co-finite, and
\item \textit{chaotic} if it is hypercyclic with a dense subspace of periodic vectors. (A vector $u$ is periodic for $T$ if its orbit returns to $u$, i.e., $T^ku=u$ for some $k$.)
 \end{itemize}
 \end{definition}

The transitivity and hypercyclicity are equivalent properties in separable Banach spaces. The aforementioned dynamical properties of operators are extensively studied for the class of weighted shift operators. Recall that the weighted backward shift on an analytic function space over the unit disc in the complex plane is  given by
\begin{center}
$
 B_w\big(\sum_{n=0}^{\infty}\lambda_nz^n\big):=\sum_{n=0}^{\infty}w_{n+1}\lambda_{n+1}z^n.
$
\end{center}
These can also be analogously defined on sequence spaces. In fact, in the study of linear dynamical systems, weighted shift operators on sequence spaces have been investigated extensively, whereas weighted shifts on familiar analytic function spaces did not receive considerable attention. Motivated by this, we study the hypercyclicity of weighted shifts on weighted Bergman spaces corresponding to both radial and non-radial weights. For the subject matter of linear dynamics, we refer to the monographs by Bayart and Matheron \cite{Bayart-Matheron} and Grosse-Erdmann and Peris \cite{Erdmann-Peris}. Also, for the theory of classical weighted shifts, see Halmos \cite{Halmos} and Shields \cite{Shields1}. We will make use of the following standard criteria in linear dynamics for establishing the hypercyclicity and chaos of weighted shifts on weighted Bergman spaces. For more and variants of these, we refer to \cite{Bayart-Matheron} and \cite{Erdmann-Peris}.

\begin{theorem} [\textsf{Gethner-Shapiro Criterion}, cf. \cite{Gethner-Shapiro}]\label{thm-hypc}
Let $T$ be a bounded operator on a separable Banach space $X$, and let $X_0$ be a dense subset of $X$. If $\{n_k\} \subseteq \mathbb{N}$ is a strictly
increasing sequence and $S:X_0\mapsto X_0$ is a map such that, for each $x\in X_0,$ $\lim_{k\rightarrow \infty} T^{n_k}x=0=\lim_{k\rightarrow \infty}S^{n_k}x$
and $TSx=x,$
then $T$ is hypercyclic. Moreover, if $n_k=k$ for all $k\geq 1$, then $T$ is mixing on $X$.
\end{theorem}
A similar criterion, known as the chaoticity criterion, has been used to obtain chaotic operators in Banach spaces. This criterion is very strong and has several important implications; see \cite{Bayart-Matheron} and \cite{Erdmann-Peris}.

\begin{theorem}[\textsf{Chaoticity Criterion}, cf. \cite{Bonilla-Erdmann1}] \label{chaos}
Let $X$ be a separable Banach space, $X_0$ be a dense set in $X$, and let $T$ be a bounded operator on $X$. If there exists a map $S:X_0\rightarrow X_0$ such that $\sum_{n\geq 0} T^nx$ and $\sum_{n\geq 0} S^nx$
are unconditionally convergent, and $
TSx=x
$ for each $x\in X_0$, then the operator $T$ is chaotic and mixing on $X$. 
\end{theorem}

There is yet another powerful criterion:

\begin{theorem} [\textsf{Eigenvalue Criterion}, cf. \cite{Bayart-Matheron} or \cite{Erdmann-Peris}]
    Suppose $X$ is a separable, complex Banach space. If $T$ is a bounded operator on $X$ such that span $\{x\in X:Tx=\lambda x,~ |\lambda|<1\}$ and span $\{x\in X:Tx=\lambda x, ~|\lambda|>1\}$ are dense in $X$, then $T$ is hypercyclic. In addition, if span $\{x\in X: Tx=\lambda x,~\lambda^n=1, n\in \mathbb{N}\}$ is dense, then $T$ is chaotic.
\end{theorem}

Dynamics of shifts in the context of Banach spaces was initiated by Rolewicz \cite{Rolewicz} who showed that $\lambda B$ is hypercyclic on the sequence space $\ell^p$, where $|\lambda|>1$, $B$ is the unweighted backward shift, and $1\leq p<\infty$. Kitai \cite{Kitai} provided a general criterion for the hypercyclicity of continuous operators on Banach spaces, and gave a class of hypercyclic shifts. Gethner and Shapiro \cite{Gethner-Shapiro} obtained the same criterion independently. Salas \cite{Salas-hc} provided a complete characterization of the hypercyclicity of classical unilateral and bilateral shifts. Hypercyclicity of weighted shifts on $F$-sequence spaces were characterized by Grosse-Erdmann \cite{Erdmann}. Costakis and Sambarino \cite{Costakis} characterized mixing unilateral and bilateral weighted shifts. We also refer to Beise and M\"{u}ller \cite{Beise-Muller1}, Beise, Meyrath, and M\"{u}ller \cite{Beise-Meyrath-Muller}, Bourdon and Shapiro \cite{Bourdon-Shapiro}, and M\"{u}ller and Maike \cite{Muller-Maike} for the dynamics related to the backward shift on the unweighted Bergman spaces of some simply connected domains. We mention that Gethner and Shapiro (\cite{Gethner-Shapiro}, p. 287) established the hypercyclicity of the unweighted backward shift on the Bergman space $A^2(\mu)$ for certain radial measures $\mu$ supported on $[0,1)$.

The paper is organized as follows. In Section $2,$ we obtain coefficient estimates in $A^p_{\phi}$ for certain weights $\phi(z)$, and obtain upper bounds for the norms of coefficient functionals $k_n$. (These estimates will be used in studying the continuity and hypercyclicity of shifts.) Continuity of weighted shifts is studied in Section $3$, and specifically, we prove that $B$, the unweighted shift, is continuous on $A^p_{\phi}$ for any integrable weight $\phi$. Further, the hypercyclicity, mixing, and chaos for $B_w$ will be established in Section $4$. Examples of non-radial Bergman spaces is considered in Section $4$.

\section{Norm estimates for coefficient functionals on $A^p_{\phi}$}

In this section, we first introduce the weighted Bergman space $A^p_{\phi}$, and as our main results, we obtain norm estimates for point evaluations and the $n$-th coefficient functional $k_n$, given by
\[
k_n(f):=\frac{f^{(n)}(0)}{n!}, \hskip .7cm n\geq 0,
\]
on $A^p_{\phi}$ for certain weights. By a \textit{weight function} we mean a strictly positive and integrable function $\phi(z)$ defined on the unit disc. For a given weight $\phi(z)$, and $1\leq p<\infty,$ the weighted Bergman space $A_{\phi}^{p}$ consists of analytic functions $f(z)$ on the unit disc such that
\[
\lVert f\rVert_{A^{p}_{\phi}}^{p}:=\int_{\mathbb{D}}\lvert f(z)\rvert^{p}\phi( z)dA(z)<\infty,
\]
 where $dA(z):=\frac{1}{\pi}dxdy$, the normalized area measure on $\mathbb{D}$.  The weighted Bergman spaces that have been extensively investigated for decades correspond to the \textit{standard} weights $\phi(z):=(1+\alpha)(1-|z|^2)^{\alpha}$, where $\alpha>-1$, cf. Hedenmalm et al. \cite{Zhu}, and Zhu \cite{Zhu1}. The standard weights are radial, i.e. $\phi(z)=\phi(|z|)$ for all $z\in \mathbb{D}$. Let us recall the following bigger class of radial weights.

\begin{definition} A continuous radial weight $\phi(r)$ defined on $[0,1)$ is called \textit{normal} if there are constants $k>\epsilon>-1$ and $0<r_{0}<1,$ with the following properties: when $r_{0}\leq r,$ $r\to 1^{-}$, then
    \begin{equation}\label{condition9}
         \frac{\phi(r)}{(1-r)^{\epsilon}}\searrow 0~~~~~~~~~~~~\text{and}~~~~~~  \frac{\phi(r)}{(1-r)^{k}}\nearrow \infty.
    \end{equation}
    \end{definition}
The original definition of normal weights, due to Shields and Williams \cite{Shields2}, p. 291, demands an extra condition that $\lim_{r\rightarrow 1}\phi(r)=0$, but this is not required for our purpose, and so, we did not assume this condition in the above definition. (In that paper, the authors studied mainly about Bergman projections as well as the duality in weighted Bergman spaces.) There are several classes of normal weights appearing naturally, but we consider the following only:
\begin{itemize}
    \item[(i)] ~$(1+\alpha)(1-r^2)^{\alpha}$,\hskip .6cm $\alpha>-1$, 
    and
    \item[(ii)]~$(1-r)^{a}\Big(\log \frac{e}{1-r}\Big)^{b}$,\hskip .6cm $a> b> 0$.
\end{itemize}
It is not difficult to see that the above weights are normal. Weights appearing in (ii) are studied by Aleman and Siskakis \cite{Aleman2}. The Bergman spaces corresponding to the weights in (i) and (ii) are denoted by $A^p_{\alpha}$ and $A^p_{a,b}$, respectively. Also, $A^p$ will denote the unweighted Bergman space. 

Kriete and McCluare studied composition operators on weighted Bergman spaces for certain radial weights including special normal weights, cf. \cite{Kriete}. Aleman and Siskakis \cite{Aleman2} investigated integration operators (boundedness, Schatten class memberships, etc.) on $A^2_{\phi}$ for a large class of normal weights including those in (i) and (ii). In \cite{Aleman}, Aleman and Vukoti{\'c} studied the Blaschke products whose derivatives belong to Bergman spaces with normal weights. For an extensive work on weighted and unweighted Bergman spaces, we refer to Hedenmalm et al. \cite{Zhu}, J.A. Pel\'{a}ez and J. R\"{a}tty\"{a} \cite{Pelaez}, and Zhu \cite{Zhu1}. We will make use of the following point wise estimates in $A^p_{\phi}$. See, also, Lemma 2.2 of Esmaeili and Kellay \cite{Esmaeili-Kellay} for weights that are closely related to normal weights and $p=2$. 

\begin{proposition} [Aleman and Vukoti\'{c} \cite{Aleman}, p. 496] \label{Berg}
For a normal weight $\phi(z)$ and $1\leq p<\infty$, there exists a constant $C_p$ (depending only on $p$) such that
\[
     \lvert f(z)\rvert\leq \frac{C_{p}}{(1-\lvert z\rvert^{2})^{\frac{2}{p}}\phi(\lvert z\rvert)^{\frac{1}{p}}}\lVert f \rVert_{A^{p}_{\phi}},
\]
for all $f\in A^p_{\phi}$ and $z \in \mathbb{D}$.

\end{proposition}

We now obtain norm estimates for coefficient functionals $k_n$ defined on $A^p_{\phi}$. This estimate is simple, at the same time it is of much use in our work. Indeed, we will utilize this estimate to obtain necessary conditions for $B_w$ to become hypercyclic, or to have periodic vectors, cf. Section $4$. The continuity of certain weighted shifts will be based on this estimate, cf. Sect. $3$.

\begin{proposition}\label{Bergman1}
    If $\phi(z)$ is a radial weight (not necessarily normal), then there exists a constant $C$ such that
    \[
\frac{\left|f^{(n)}(0)\right|}{n!}\leq C\Big(\int_0^{1}r^{np+1}\phi(r)dr\Big)^{-1/p}\|f\|_{A^p_{\phi}},~ n\geq 0,
    \]
    for all $f\in A^p_{\phi}$. In particular, we have $\|k_n\|\leq C \frac {1} {\|z^n\|_{A^p_{\phi}}},$ for all $n \geq 0$.
\end{proposition}
\begin{proof}
For $f\in A^p_{\phi}$, and $r\in(0,1)$, using Cauchy's integral formula, we find that
\begin{equation}\label{rad-coeff}
r^n\frac{|f^{(n)}(0)|}{n!}\leq \big(\int_{-\pi}^{\pi}|f(re^{i\theta})|^pd\theta\big)^{1/p}.
\end{equation}
Since
\begin{eqnarray*}
          \left\lVert  z^{n} \right\rVert_{A^{p}_{\phi}}^{p}   &=& \int_{\theta=0}^{2\pi}\int_{r=0}^{1} r^{pn}\phi(r)r\frac{drd\theta}{\pi}
          =2\int_{0}^{1} r^{np+1} \phi(r)dr,
        \end{eqnarray*}
a simple adjustment in \eqref{rad-coeff} using $\phi(r)$ gives that
\[
\|z^n\|_{A^p_{\phi}}\frac{|f^{(n)}(0)|}{n!}\leq C\|f\|_{A^p_{\phi}}.
\]
This is the desired estimate in the proposition.
\end{proof}


Though our study on the dynamics of $B_w$ is mostly concentrated on the (radial) normal weighted Bergman spaces, we will also consider the following non-radial examples. The second one is subharmonic in $\mathbb{D}$, and $g(z)$ will be assumed to be a univalent function, having strictly positive real part.
\begin{itemize}
    \item[(i)] ~$1-|\Re(z)|$,\hskip .6cm $\alpha>-1$, 
    and
    \item[(ii)]~$|g(z)|^{\gamma}$,\hskip .6cm $0<\gamma<1$.
\end{itemize}
\noindent The reason for taking the range $0<\gamma<1$ is that the polynomials are dense in the corresponding weighted Bergman space, which we require in deducing the hypercyclicity of shifts via Gethner-Shapiro criterion. Weights close to (i) have appeared in the study of shift-invariant subspaces, due to Richter (cf. \cite{Richter}, p. 329). We will now derive coefficient estimates in $A^p_{\phi}$, in the spirit of Proposition $2.3$, and for this purpose, we recall the sub-mean value property of subharmonic functions: Let $\Omega$ be an open and connected set in $\mathbb{C}$. A function $u:\Omega\rightarrow \mathbb{R}$ is said to be \textit{subharmonic} if it is upper semi-continuous, and has the sub-mean value property
\begin{equation}\label{condition19}
u(z_0)\leq \frac{1}{2\pi}\int_{0}^{2\pi}u(z_0+Re^{i\theta})d\theta,
\end{equation}
for every closed disc $|z-z_0|\leq R$, contained in $\Omega$. The above integral can be converted to an area integral (cf. \cite{Zhu}, p. 3), and then (\ref{condition19}) becomes
\[
    u(z_0)\leq \frac{1}{R^{2}}\int_{|z-z_0|\leq R}u(z)dA(z).
\]
Along with the above expression, we will make use of the well known fact: if $f(z)$ is analytic on some domain, and $p>0$, then $|f(z)|^p$ is subharmonic.

\begin{proposition}\label{non-radial}
Let $\phi(z)=1-|\Re(z)|$. The following hold for $f\in A^p_{\phi}$, where $1\leq p<\infty$.
\begin{enumerate}
    \item[(i)] There exists a constant $C_p>0$ such that
    \begin{equation}\label{condition12}
         |f(z)|\leq \frac{C_{p}}{(1-|z|)^{\frac{2}{p}}(1-|\Re(z)|)^{\frac{1}{p}}}\lVert f\rVert_{A^{p}_{\phi}}, ~z\in \mathbb{D}.
    \end{equation}
    \item[(ii)] For $n\geq 0$, there exists a constant $C_{p}^{'}>0$ such that \begin{equation}\label{conditon14}
      \frac{\left|f^{(n)}(0)\right|}{n!}\leq C_{p}^{'}(n+1)^{\frac{3}{p}}\lVert f \rVert_{A^{p}_{\phi}}.
\end{equation}
\end{enumerate}
\end{proposition}

\begin{proof}
 For any $ z\in \mathbb{D},$ let $R=\frac{1-\lvert z\rvert}{2},$ and $f\in A^{p}_{\phi},$ then it follows from the subharmonicity of $\lvert f(z)\rvert^{p}$ in $\mathbb{D}$ that 
 \begin{eqnarray*} 
     \lvert f(z)\rvert^{p}&\leq& \frac{4}{(1-\lvert z \rvert)^{2}}\int_{B(z,R)}\lvert f(\zeta)\rvert^{p}dA(\zeta)\\
     &=& \frac{4}{(1-\lvert z \rvert)^{2}}\int_{B(z,R)}\lvert f(\zeta)\rvert^{p}\frac{1-|\Re(\zeta)|}{1-|\Re(\zeta)|}dA(\zeta).
 \end{eqnarray*}
Now, if $|z-\zeta|<R=\frac{1-|z|}{2},$ then we get
 \[
 |\Re(\zeta)|-|\Re(z)|\leq |\Re(z)-\Re(\zeta)|=|\Re(z-\zeta)|\leq |z-\zeta|<\frac{1-|z|}{2}\leq \frac{1-|\Re(z)|}{2},
 \]
 which implies that $\frac{1}{1-|\Re(\zeta)|}\leq \frac{2}{1-|\Re(z)|}.$
 Thus,
 \begin{eqnarray*}
      \lvert f(z)\rvert^{p} &\leq& \frac{8}{(1-\lvert z \rvert)^{2}(1-|\Re(z)|)}\int_{B(z,R)}\lvert f(\zeta)\rvert^{p}(1-|\Re(\zeta)|)dA(\zeta)\\
      &\leq&\frac{8}{(1-\lvert z \rvert)^{2}(1-|\Re(z)|)}\int_{\mathbb{D}}\lvert f(\zeta)\rvert^{p}\varphi(\zeta)dA(\zeta),
 \end{eqnarray*}
and hence, we get the first part (i).
    
Applying the Cauchy integral estimates and utilizing the previously established point-wise bound, we obtain 
 \begin{equation}\label{condition13}
     \frac{\left|f^{(n)}(0)\right|}{n!}\leq \frac{1}{R^{n}} \max_{|\xi|\leq R}|f(\xi)|\leq \frac{1}{R^{n}}\frac{C_{p}}{(1- R)^{\frac{3}{p}}}\lVert f \rVert_{A^{p}_{\phi}}.
 \end{equation}
Since for $-\pi \leq \theta < \pi$, the function $1 - R|\cos\theta|$ attains its minimum value $1-R$. 
 Now, consider the function $g(R)= R^{np}(1-R)^{3},$ defined on the interval $[0,1]$. Using differentiation, we get the maximum of $g(R)$ at $\frac{np}{np+3},$ with maximum being given by $(\frac{np}{np+3})^{np}\frac{27}{(np+3)^{3}}.$ Since
           \[
\frac{(np+3)^{np}}{(np)^{np}}=\left(1+\frac{3}{np}\right)^{np}\to e^{3},~~~~(n\to\infty),
\]
by combining \eqref{condition13} and the above maximum value, we obtain the estimate in (ii), as desired.
\end{proof}

Our next results give similar estimates in $A^p_{\phi}$, where $\phi(z)=|g(z)|^{\gamma}$ for a univalent function $g(z)$. 

  \begin{proposition}\label{Noon-radial2}
Let $g(z)$ be a univalent analytic function on $\mathbb{D}$, and let $0<\gamma<1$. Then the following statements hold for the weight $\phi(z):=|g(z)|^{\gamma}$, and $f\in A^p_{\phi}$, where $1 \leq p < \infty$.
\begin{enumerate}
    \item[(i)] There exists a constant $C_p > 0$ such that 
    \begin{equation}\label{non-radial3}
    |f(z)|
    \le 
    \frac{C_p}{(1-|z|)^{\frac{2}{p}}|g(z)|{^{\gamma}/p}}
    \|f\|_{A^{p}_{\phi}},
        \qquad z \in \mathbb{D}.
    \end{equation}
    
    \item[(ii)] For each $n \geq 0$, there exists a constant $C_{p}^{\prime} > 0$ such that 
    \begin{equation}\label{non-radial4}
         \frac{|f^{(n)}(0)|}{n!} \leq \frac{C_p \|f\|_{A^p_{\phi}}}{r^n(1 - r)^{\frac{2}{p}}} \int_0^{2\pi} \frac{1}{ |g(r e^{i\theta})|^{\frac{\gamma}{p}}} d\theta,\quad 0<r<1.
    \end{equation}
\end{enumerate}
\end{proposition}

\begin{proof}
Let $f\in A^p_{\phi}$. Note that the function 
\[
|f(z)|^{p} |g(z)|^{\gamma}, ~z\in \mathbb{D},
\]
is subharmonic, as it is a power of the modulus of $f(z)g(z)^{\gamma/p}$. Let $z \in \mathbb{D},$ and $r = \frac{1 - |z|}{2}.$ We obtain that
\begin{eqnarray*}
    |f(z)|^{p}
    |g(z)|^{\gamma}
    &\leq & 
    \frac{4}{(1-|z|)^{2}}
    \int_{B(z,r)} 
    |f(\zeta)|^{p}
    |g(\zeta)|^{\gamma}
    dA(\zeta)\\&\leq &\frac{4}{(1-|z|)^{2}}\|f\|_{A^p_{\phi}}^p,
\end{eqnarray*}
which gives the part (i).

\smallskip

To prove (ii), we rely on the Cauchy's integral formula.  
Using the previously derived point wise bound, we have 
\begin{equation*}
    \frac{|f^{(n)}(0)|}{n!} \leq \frac{1}{2\pi} \int_{|\xi|=r} \frac{|f(\xi)|}{|\xi|^{n+1}} |d\xi| \leq\frac{C_p^{'} \|f\|_{A^p_{\phi}}}{r^n(1 - r)^{\frac{2}{p}}} \int_0^{2\pi} \frac{1}{ |g(r e^{i\theta})|^{\frac{\gamma}{p}}} d\theta.
\end{equation*}
The proof is complete.
\end{proof}

We will need the asymptotic behavior of certain integrals, as given in the next proposition. These will be utilized in Sections $3$ and $4$. Let us recall an elementary result on the gamma function $\Gamma(x)$, for $x>0$, which can be proved using the Stirling formula
$\Gamma(x+1)\sim \sqrt{2\pi x} \left(\frac{x}{e}\right)^{x},$ for large $x$, cf. Rudin \cite{Rudin}, Chapter 8. Here, the symbol $\sim$ means that the ratios tend to $1$, asymptotically. Also, two functions $J(x)$ and $K(x)$ defined on $[a,b)$ satisfy 
\[
J(x)\asymp K(x),\quad x\rightarrow b^{-}, 
\]
if $J/K$ is bounded above and below in a small interval $(a+\delta, b)$. 

The proof of the lemma is omitted.

       \begin{lemma}\label{gamma}
           For real numbers $t$ and $x$, we have
           \[
           \frac{\Gamma(t+n)}{\Gamma(x+n)}\sim n^{t-x},\quad \text{for large } n.
           \]
       \end{lemma}

With this lemma, we obtain the following. (We refer to Zhan et al. \cite{Zhan}, p. 3, for similar types of integrals without proofs.)

\begin{proposition}\label{integration-pro}
    
The following hold for $p>0.$
\begin{itemize}
\item[(i)] If $\alpha>-1$, then
$\int_{0}^{1}r^{pn+1} (1 - r^{2})^{\alpha}dr\sim \frac{1}{(n+1)^{\alpha+1}},\quad \text{n}\rightarrow \infty.
$
\item[(ii)] If $a,b>0,$ then 
\[
\int_{0}^{1} r^{pn+1}(1-r)^{a}\Big(\log\frac{e}{1-r}\Big)^{b} dr\asymp \frac{(\log(n+1))^{b}}{(n+1)^{a+1}},\quad \text{n}\rightarrow \infty.
\]
\end{itemize}
\end{proposition}
\begin{proof} The integral in (i) is precisely, equal to
\[
       \frac{\Gamma\!\left(\tfrac{pn}{2} + 1\right)
             \Gamma(\alpha + 1)}
            {\Gamma\!\left(\tfrac{pn}{2} + \alpha + 2\right)}.
\]

\noindent Now, by applying Lemma~\ref{gamma}, we obtain the required asymptotic estimate.

Now, we prove (ii). Denote the integral in (ii) by $I_n$. Make a change of variable $t = 1 - r$. Then
$
I_n = \int_{0}^{1} (1-t)^{pn+1} t^{a} \left(\log(e/t)\right)^{b} \, dt.
$ Consider $t \in [0, \tfrac{1}{n+1}]$.  
For such $t$, we have $(1-t)^{pn+1} \ge (\frac{n}{n+1})^{pn+1}.$ Therefore,
\[
I_n \ge \left(\frac{n}{n+1}\right)^{pn+1} 
\int_{0}^{1/(n+1)} t^{a} \left(\log\frac{e}{t}\right)^{b} \, dt.
\]
The integral in the right side of the above inequality is lower bounded by $(\int_0^{1/(n+1)} t^adt)[\log(n+1)]^b$. Since $\left(\frac{n}{n+1}\right)^{pn+1} \to e^{-p}$ as $n \to \infty$, 
we conclude that, there is a constant $C>0$ such that
$
I_n \geq C \frac{(\log(n+1))^{b}}{(n+1)^{a+1}},
$
for large $n$.

For the reverse inequality, observe that
\begin{eqnarray*}
    \int_{0}^{\frac{n}{n+1}} r^{pn+1}(1-r)^{a}\Big(\log\frac{e}{1-r}\Big)^{b} dr&\leq& (\log(e(n+1)))^{b}\int_{0}^{1} r^{pn+1}(1-r)^{a}dr
    \sim\frac{(\log(n+1))^{b}}{(n+1)^{a+1}},
\end{eqnarray*}
for large $n.$ 

\end{proof}

 \section{Continuity of weighted shifts on $A^p_{\phi}$}

 In this section, we obtain sufficient conditions for the weighted shifts $B_w$ to be continuous on $A^p_{\phi}$ corresponding to $w=\{w_n\}_{n=1}^{\infty}$,  where the weight $\phi$ is either normal or the specific non-radial examples considered in Section $2$. It is known that the unweighted shift $B$ is bounded on the unweighted Bergman space $A^p$ (\cite{Zhu} or \cite{Zhu1}). Denote the coefficient multiplier operator $z^n\mapsto w_nz^n$, ($n\geq 0$, $w_0=0$) acting on $A^p_{\phi}$ by $M_w$. For $M_w$, we have assumed that the weight sequence starts with $w_0=0$. Then, it follows that $B_w=BM_w$, and so, if $B$ and $M_w$ are continuous on $A^p_{\phi}$, then so is $B_w$. We now provide conditions for $B$ and $M_w$ to be continuous.

It is easy to see that $A^2_{\phi}$ can be regarded as a weighted sequence space as $\{z^n/\|z^n\|:n\geq 0\}$ is an orthonormal basis in $A^2_{\phi}$. It is easy to check that $B_w:A^2_{\phi}\rightarrow A^2_{\phi}$ is unitarily equivalent to a standard weighted backward shift $B_{\lambda}$ on the sequence space $\ell^2$, where $\lambda=(\lambda_n)_{n\geq 1}$, and
\begin{equation}\label{new-weight}
\lambda_n:=w_n \frac{\|z^{n-1}\|}{\|z^n\|},~n\geq 1.
\end{equation}
Note that 
\begin{equation}\label{new-weightt}
\|z^n\|_{A^2_{\phi}}=\Big (2\int_{0}^{1}\phi(r)r^{2n+1}dr\Big)^{1/2}, ~n\geq 0.
\end{equation}
Thus, in the case of $p=2$ and radial $\phi(z)$, most of the dynamical and some operator theoretic aspects of $B_w$ follow from the usual weighted shifts. In particular, $B_w$ is continuous on $A^2_{\phi}$ if and only if the above sequence $\{\lambda_n\}$ is bounded. However, for an arbitrary $p\geq 1$, the continuity (and dynamics) of $B_w$ on $A^p_{\phi}$ is not straightforward to understand even in the unweighted case of $A^p$. Indeed, these involve the study of coefficient multipliers in $A^p_{\phi}$. The notion of coefficient multipliers between various analytic function spaces is a classical theme in functional analytic complex analysis. We refer to the book by Jevti{\'c}, Vukoti{\'c}, and Arsenovi{\'c}~\cite{Jevtic}; see, also, Blasco~\cite{Blasco}, Buckley, Ramanujan, and Vukoti{\'c}~\cite{Buckley}, Shields and Williams~\cite{Shields2}, and Vukoti{\'c}~\cite{Dragan} for discussions on coefficient multipliers of the standard weighted Bergman spaces.

 We first prove the continuity of $B$ in a slightly general set up. Recall that the weights considered in this paper are integrable.
 
\begin{proposition}\label{bounded1}
The unweighted backward shift $B$ is continuous on $A^{p}_{\phi}$, $1\leq p<\infty$.
 \end{proposition}
 \begin{proof}
 For $f\in A^p_{\phi}$,  note that 
 \[
 Bf(z)=\begin{cases}
\frac{f(z)-f(0)}{z},& z\neq 0\\
f^{\prime}(0), & z=0.
\end{cases}
\] 
Fix a small closed ball $\overline{B(0,\delta)}\subseteq \mathbb{D}$. As $\phi(z)$ is integrable and $Bf(z)$ is continuous on the above closed ball, we have $\int_{B(0,\delta)} |Bf(z)|^p \phi(z)dA(z)<\infty$. Also, in the region $\mathbb{D}\setminus B(0,\delta) $ we have 
\[
|Bf(z)|\leq \frac{1}{\delta}|f(z)-f(0)|.
\]
The function $f(z)-f(0)$ lies in $A^p_{\phi}$ because the space contains constants. It follows that\\ $\int_{\mathbb{D}\setminus B(0,\delta)} |Bf(z)|^p \phi(z)dA(z)<\infty.$ Thus, we get that $Bf\in A^p_{\phi}$. By the well-known closed graph theorem, $B$ is bounded.
 \end{proof}

 Next, we turn our attention to study the continuity of a coefficient multiplier $M_w$ on weighted Bergman spaces with emphasis on $A^p_{\alpha},$ $A^p_{a,b},$ and non-radial weights consider in Section 2. We will obtain our results using the norm estimates of the coefficient functional $k_n$, which we derived in Propositions \ref{Bergman1}, \ref{non-radial}, and \ref{Noon-radial2}.

  \begin{theorem}
 The following hold for a weight sequence $w=\{w_n\}_{n=0}^{\infty}$, and $1\leq p<\infty$.
 \begin{itemize}
\item[(i)] Let $\phi(z)$ be a normal weight, and let
\[
\tilde{M}(r):=\sum_{n\geq 1} |w_n|\Big(\int_0^{1}R^{np+1}\phi(R)dR\Big)^{-1/p} r^n.
\]
If there exist constants $C_1>0$ and $-\infty<t<1/p$ such that 
\begin{equation}\label{est1}
\phi(r)^{1/p}\tilde{M}(r)\leq \frac{C_1}{(1-r)^{t}}
\end{equation}
for all $0<r<1$, then $M_w$ is a bounded operator on $A^p_{\phi}$. 
\item[(ii)] If there are constants $C_2>0$ and $-\infty<s<\frac{\alpha+1}{p}$ such that 
\[
\sum_{n\geq 1} |w_n|(n+1)^{\frac{\alpha+1}{p}} r^n\leq \frac{C_2}{(1-r)^{s}},
\]
for all $0<r<1$, then $M_w$ is bounded on $A^p_{\alpha}$.
\item[(iii)] If there are constants $C_3>0$ and $-\infty<u<\frac{a+1-b}{p}$ such that 
\[
\sum_{n\geq 1} |w_n|\frac{(n+1)^{\frac{a+1}{p}}}{(\log(n+1))^{\frac{b}{p}}} r^n\leq \frac{C_3}{(1-r)^{u}},
\]
for all $0<r<1$, then $M_w$ is bounded on $A^p_{a,b}$.
\end{itemize}
 \end{theorem}
 \begin{proof}
We need to show that, if $f(z)=\sum_{n\geq 0} \hat{f}(n)z^n\in A^p_{\phi}$, then 
$M_w(f)=\sum_{n\geq 1}w_n\hat{f}(n)z^{n}\in A^{p}_{\phi}.$ Now, recall the norm estimates of coefficient functionals from Proposition \ref{Bergman1}, i.e.  
\[
|\hat{f}(n)|\leq \frac{C}{\|z^n\|}\|f\|.
\]
Hence, for $|z|<r<1$, we have
\[
|M_wf(z)|\leq \sum_{n\geq 1}|w_n||\hat{f}(n)||z|^n\leq \tilde{M}(r)\|f\|.
\] 
From our hypothesis, it follows that, there are constants $C$ and $c$ such that
\begin{eqnarray*}
\int_{\mathbb{D}} |M_wf(z)|^{p} \phi(z) dA(z)&\leq& C\|f\|^{p}\int_{0}^{1}\int_{0}^{2\pi} (\tilde{M}(r))^{p}\phi(r)\frac{rdrd\theta}{\pi}\\
&\leq& c\|f\|^{p} \int_{0}^{1} \frac{r}{(1-r)^{pt}}dr\\
&=& c\|f\|^{p} \frac{\Gamma(2)\Gamma(-pt+1)}{\Gamma(3-pt)},
\end{eqnarray*}
which is finite as $-\infty<t<\frac{1}{p}.$ This shows that $M_w(f)\in A^p_{\phi}$.  Hence, (i) holds.

The result in \textnormal{(ii)} can be seen as follows. Indeed, using Proposition \ref{integration-pro}, we have that
$\int_{0}^{1} R^{pn+1} \phi(R)\,dR 
    \sim {(n+1)^{-\alpha - 1}},$ as $n \to \infty.
$ Hence,
$
    \tilde{M}(r) 
    \asymp \sum_{n \geq 1} 
      |w_n| (n+1)^{\frac{\alpha +1}{p}} r^n.
$
Now, using \textnormal{(i)}, we can deduce that the assertion in \textnormal{(ii)} holds.

\noindent Similarly for $\phi(r)=(1-r)^{a}(\log\frac{e}{1-r})^{b},~~a,b>0$, from the Proposition \eqref{integration-pro}, we have
\[
    \int_{0}^{1} R^{pn+1} \phi(R)\,dR 
    \asymp \frac{(\log(n+1))^{b}}{(n+1)^{a+ 1}},
    \qquad \text{as } n \to \infty.
\]
Therefore,
\[
    \tilde{M}(r) 
\asymp \sum_{n \geq 1} 
      |w_n| \frac{(n+1)^{\frac{a+1}{p}}}{(\log(n+1))^{\frac{b}{p}}} r^n.
\] Hence, (iii) holds.
 The proof is now complete.
 \end{proof}

 In the same spirit, we give sufficient conditions for $M_w$ to be continuous for the non-radial cases, below.
 \begin{theorem}
The following assertions hold, when $1\leq p<\infty$.
\begin{itemize}
    \item[(i)] For $\phi(z)=1-|\Re(z)|$, suppose there exist constants $C_{4}>0$ and $-\infty < v < \tfrac{1}{p}$ such that 
    \[
    \sum_{n\geq 1} |w_n|(n+1)^{3/p} r^n 
       \leq \frac{C_4}{(1-r)^{v}}, 
       \qquad 0<r<1.
    \]
    Then $M_{w}$ induces a bounded operator on $A^{p}_{\phi}$.
    
    \item[(ii)] For $\phi(z)=|g(z)|^{\gamma}$ with $0<\gamma<1$, suppose there exist constants $C_{5}>0$ and $k\in\mathbb{R}$ such that 
    \[
    \sum_{n\geq 1}\inf_{0<R<1}
    \left(
        \frac{\int_{0}^{2\pi}|g(Re^{i\theta})|^{-\gamma/p}\, d\theta}
            {R^{n}(1-R)^{2/p}}
    \right)
    |w_n|\, r^{n}
    \leq \frac{C_{5}}{(1-r)^{k}}
    \quad \text{and}~ \int_{0}^{1}\int_{0}^{2\pi}
    \frac{r|g(re^{i\theta})|^{\gamma}}
         {(1-r)^{pk}}
    \, d\theta\, dr < \infty,
    \]
   where $0<r<1.$ Then $M_{w}$ defines a bounded operator on $A^{p}_{\phi}$.
\end{itemize}
\end{theorem}

\begin{proof}
The argument follows exactly as in the previous theorem. 
Indeed, for $f(z)=\sum_{n\geq 0} \hat{f}(n) z^{n} \in A^{p}_{\phi},$
it follows from Proposition~\ref{non-radial} that 
\[
|M_w f(z)|
   \leq \sum_{n\geq 1} |w_n|\,|\hat{f}(n)|\,|z|^n
   \leq \lVert f\rVert\sum_{n\geq 1} |w_n|(n+1)^{\frac{3}{p}} r^{n}
   =: \tilde{N}(r)\lVert f\rVert~~(\text{say}),
\]
for $|z|<r<1.$ By the assumption, we then obtain for some constants $C^{'}$ and $c^{'}$ such that
\begin{align*}
\int_{\mathbb{D}} |M_w f(z)|^{p} \phi(z)\, dA(z)
&\leq C^{'}\lVert f\rVert^{p}\int_{0}^{1}\!\!\int_{0}^{2\pi} (\tilde{N}(r))^{p}\phi(r)\,
      \frac{r\,dr\,d\theta}{\pi} \\[3pt]
&\leq c^{'}\lVert f\rVert^{p}\int_{0}^{1}\!\!\int_{0}^{2\pi}
      \frac{r(1-r|\cos\theta|)}{(1-r)^{pv}}
      \,\frac{dr\,d\theta}{\pi} \\[3pt]
&= c^{'}\lVert f\rVert^{p} \int_{0}^{1}
      \left( \frac{2r}{(1-r)^{pv}} 
      - \frac{4r^{2}}{\pi(1-r)^{pv}} \right) dr \\[3pt]
&=c^{'}\lVert f\rVert^{p} \left(
      \frac{2\,\Gamma(2)\Gamma(1-pv)}{\Gamma(3-pv)} 
      - \frac{4\,\Gamma(3)\Gamma(1-pv)}{\Gamma(4-pv)} 
      \right).
\end{align*}
The last expression is finite whenever $-\infty < v < \tfrac{1}{p}$. 
Hence $M_{w}(f)\in A^{p}_{\phi}$, and the operator $M_{w}$ is bounded on $A^{p}_{\phi}$. Hence, (i) holds.\\
Similarly, for $\phi(z)=|g(z)|^{\gamma}$ with $0<\gamma<1$, Proposition~\ref{Noon-radial2} yields
\[
|M_{w} f(z)|
   \leq \lVert f\rVert\sum_{n\geq 1}\inf_{0<R<1}
    \left( 
       \frac{\int_{0}^{2\pi}|g(Re^{i\theta})|^{-\gamma/p}\, d\theta}
            {R^{n}(1-R)^{2/p}}
    \right)
    |w_n|\, r^{n}.
\]
Along with the stated assumptions, this implies that $M_{w}$ is a bounded operator on $A^{p}_{\phi}$.

\end{proof}

Applying the previous two theorems, we can get multipliers which are new in $A^p_{a,b}$, and also, in the non-radial cases.

 \begin{example}
(i) The sequence $w_n = \left({n+1}\right)^{-\frac{\alpha+1}{p}}, \quad n \ge 1,$ defines a bounded multiplier $M_w$ on $A^p_{\alpha}$; consequently, the operator $B_w$ is bounded on $A^p_{\alpha}$. 

(ii) If we take 
\[
w_n = \frac{(\log(n+1))^{\frac{b}{p}}}{(n+1)^{\frac{a+1}{p}}}, \quad n \ge 1,
\]
then $M_w$ is a bounded multiplier on $A^p_{a,b}$, and hence $B_w$ is bounded on $A^p_{a,b}$.
 
(iii) Consider the special case of $A^1$. Since the multiplier corresponding to $w_n=-1$ for $n\geq 1$, is continuous, from (i) it follows that the weight sequence
$
w_n:=\frac{n}{n+1},~n\geq 1,
$
gives a bounded multiplier, and hence, $B_w$ corresponding to this weight is bounded on $A^1$.
\end{example}

\begin{example}
The sequence $w_n = ({n+1})^{-\frac{3}{p}},$ $n \ge 1,$ defines a bounded multiplier $M_w$ on the weighted Bergman space $A^p_{\phi}$, where $\phi(z) = 1 - |\Re(z)|$. Similarly, for any $-\infty<k<\tfrac{1-\gamma}{p}$, if we choose  $g(z)=\tfrac{1+z}{1-z}$ and $w_{n}=(n+1)^{-(2+\gamma)/p}$ for $n\geq 1$, 
then $M_{w}$ acts as a multiplier on $A^{p}_{\phi}$ corresponding to 
$\phi(z)=\left|\tfrac{1+z}{1-z}\right|^{\gamma}$.
\end{example}
\section{Hypercyclicity and chaos}

In this section, we find necessary and sufficient conditions for $B_w$ to be hypercyclic, mixing, and chaotic on $A^p_{\phi}$ when $\phi$ is a normal weight or certain subharmonic weight. \textsf{For this section, we will assume that all weight sequences $w=\{w_n\}$ are having non-zero terms.} To apply the Gethner-Shapiro criterion in $A^p_{\phi}$, we need the denseness of polynomials, as given below, for integrable weights. This is well known for $A^p_{\alpha}$, $\alpha>-1$ (cf. Hedenmalm et. al. \cite{Zhu}), $A^1_{\phi}$ for special normal weights (cf. Shields and Williams \cite{Shields2}). Originally, the denseness of polynomials in case the weight is radial, goes back to Mergelyan \cite{Mergelyan}. Also, see the works of Abkar \cite{Abkar} and Pelaez \cite{Pelaez} for certain non-radially weighted Bergman spaces. The proof of the following result is included for completeness purpose.
 
\begin{proposition}\label{P}
Let $\phi(z)$ be a (not necessarily radial) weight, and $1 \leq p < \infty$. Then, the space of all polynomials is dense in the weighted Bergman space $A^{p}_{\phi}$, $1\leq p<\infty$.
\end{proposition}

\begin{proof}
Let $f \in A^{p}_{\phi}$, and define its dilations by $f_{r}(z) = f(rz)$ for $0 < r < 1$.  
Each function $f_{r}$ is analytic on a larger disk $|z|<1/r$, so it can be uniformly approximated on $\overline{\mathbb{D}}$ by polynomials, i.e., by the partial sums of its Taylor series.  
Thus, it suffices to show that $f$ can be approximated in the $A^{p}_{\phi}$-norm by its dilatation:
$\lVert f - f_{r} \rVert_{p}^{p} \to 0 \quad \text{as } r \to {1^{-}}.$

For $\varepsilon > 0$, we can choose $\delta$ such that
$
\int_{|z| > \delta} |f(z)|^{p} \phi(z) \, dA(z) < \varepsilon,
$
and $|f(z) - f(rz)| < \varepsilon,~ \text{for}~~ |z| \le \delta,$ and  $r$ which is sufficiently close to $1$. Since $|f(rz)| \le |f(z)|$, it follows that
$
\int_{|z| > \delta} |f(rz)|^{p} \phi(z) \, dA(z) < \varepsilon.
$
Now, consider the inequality
\begin{equation}\label{den}
    \lVert f - f_{r} \rVert_{p}^{p}
    \leq \int_{|z| \leq \delta} |f(z) - f(rz)|^{p} \phi(z) \, dA(z)
      + \int_{|z| > \delta} (|f(z)| + |f(rz)|)^{p} \phi(z) \, dA(z).
\end{equation}
Then,
\[
\int_{|z| \le \delta} |f(z) - f(rz)|^{p} \phi(z) \, dA(z)
    < \varepsilon^{p} \int_{|z| \le \delta} \phi(z) \, dA(z),
\]
and the second integral in \eqref{den} is at most $2^p \varepsilon$. Combining the two estimates, we see that the family $f_r$, $0<r<1$ approximates $f$, completing the proof.
\end{proof}

Two applications are given below.

\begin{corollary}\label{Herg}
Let $\phi(z):=|g(z)|^{\gamma}$, where $0<\gamma<1$, and $g(z)$ is a univalent analytic function on $\mathbb{D}$, having positive real part, and $g(0)=1$. Then, the polynomials are dense in $A^p_{\phi}$, $1\leq p<\infty$.
\end{corollary}
\begin{proof}
By the well known Herglotz representation for positive harmonic functions on $\mathbb{D}$, there exists a probability measure $\mu$ on $[0,2\pi)$, such that
\[
g(z)=\int_{0}^{2\pi} \frac{e^{i\theta}+z}{e^{i\theta}-z} d\mu(\theta),
\]
$z\in \mathbb{D}$. Hence, we find that 
\[
|g(z)|\leq C \frac{1}{1-|z|},
\]
for some constant, independent of $z$. It can now be seen that $|g(z)|^{\gamma}$ is integrable on $\mathbb{D}$, and so, the polynomials are dense in $A^p_{\phi}$ by the preceding proposition.
\end{proof}

See, also, \cite{Pelaez}, Ch. $1$, where the authors obtained the previous result for more general univalent functions using a different approach.

\begin{corollary}
The polynomials are dense in $A^p_{\phi}$, $1\leq p<\infty$, where $\phi(z)=1-|\Re(z)|$, $z\in \mathbb{D}$.
\end{corollary}
\begin{proof}
This follows at once from Proposition \ref{P} as the weight is integrable.
\end{proof}

 The above result can, also, be obtained alternatively. Indeed, the inequality $1-r\leq 1-r|\cos \theta |\leq 1$, readily implies that the inclusions $A^p\hookrightarrow A^p_{\phi}\hookrightarrow A^p_1$ are continuous. Since the polynomials are dense in $A^p$ and $A^p_1$, we get the same result for $A^p_{\phi}$, where $A^p_1$ is the weighted Bergman space corresponding to $1-r$.

\subsection{Shifts on radially weighted Bergman spaces}
We characterize hypercyclic and mixing weighted shifts on $A^p_{\phi}$ for a radial weight $\phi$. In the Corollaries \ref{Hyper} and \ref{Hyperc}, we give complete characterizations for the hypercyclicity of $B_w$ on $A^p_{\phi}$ having the standard as well as logarithmic type weights, respectively. In particular, we prove that the unweighted shift $B$ is mixing on $A^p_{\phi}$, $1\leq p<\infty$, for all normal weights.

A general result on the hypercyclicity of $B_w$, which works even in non-radial cases, is given below.

\begin{theorem}\label{bsaf}
Let $\mathcal{E}$ be a separable Banach space of analytic functions on $\mathbb{D}$, with the following properties:
\begin{itemize}
\item[(E)] The evaluation functional $ev_z: f\mapsto f(z)$ is continuous at every $z\in \mathbb{D}$,
\item[(P)] the polynomials are dense in $\mathcal{E}$, and
\item[(C)] the $n$-the coefficient functional $k_n$ satisfies $\sup_{n\geq 0}\|k_n\|_{\mathcal{E}^*} \|z^n\|_{\mathcal{E}}<\infty$. 
\end{itemize}
If $B_w$ is continuous on $\mathcal{E}$, then $B_{w}$ is hypercyclic on $\mathcal{E}$ if and only if 
\[
\displaystyle \liminf_{n\rightarrow \infty}~\left(\frac{1}{|w_1\cdots w_n|}\|z^n\|_{\mathcal{E}}\right)=0.
\]
Also, $B_{w}$ is mixing on $\mathcal{E}$ if and only if 
\[
\displaystyle \lim_{n\rightarrow \infty}~\left(\frac{1}{|w_1\cdots w_n|}\|z^n\|_{\mathcal{E}}\right)=0.
\]
\end{theorem}
\begin{proof}

If $B_{w}$ is hypercyclic on $\mathcal{E}$, having a hypercyclic vector $f(z)$, then the denseness of the orbit
$\{(B_{w})^{n}f:n\geq1 \}$, along with the continuity of some (non-zero) coefficient functional $k_j$ implies that 
\[
\{k_j(B_w^nf): n\geq 1\}
\]
is dense in $\mathbb{C}$. Now, computing the $j$-th coefficient of $B_w^n(f)$, we get that
   \begin{equation}\label{deri}
   \sup_{n\geq 1}~\left\lvert w_{j}w_{j+1}\cdots w_{n+j-1} \frac{f^{(n+j-1)}(0)}{(n+j-1)!}\right\rvert=\infty.
   \end{equation}
By making use of the property (C), we find that 
\[
\sup_{n\geq 1}\frac{|w_1\cdots w_n|}{\|z^n\|_{\mathcal{E}}}=\infty,
\]
as we wanted in the theorem. For the converse part, we recall and apply the Gethner-Shapiro criterion. Let $X_0$ be the space of all polynomials, which is dense in $\mathcal{E}$. Consider the forward weighted shift $S:X_0\rightarrow X_0$  given by $S(z^{j})= \frac{z^{j+1}}{w_{j+1}},~\hspace{.0cm} j\geq 0.$
     Trivially, 
$B_{w}Sf=f$, and $(B_{w})^nf\rightarrow 0$, as $~n\rightarrow \infty,
$
for all $f\in X_0$. It suffices to show that, there exists a strictly increasing sequence $\{m_k\}$ of natural numbers such that
\[
 \|S^{m_k}(z^j)\|_{\mathcal{E}}\rightarrow 0,
\]
as $k\rightarrow \infty$, for every monomial $z^j$. Note that
\[
S^n(z^j)=\frac{1}{w_{j+1}w_{j+2}\cdots w_{j+n}} z^{j+n}.
\]
Hence, for $p\geq 1,$ and by using the assumption in (ii), we get an increasing sequence $\{d_k\}$ such that $ S^{d_k}(z^n)\rightarrow 0,$ as $k\rightarrow \infty$. Now, Lemma 4.2 of \cite{Erdmann-Peris} yields a sequence $\{m_k\}$, required as above. 

For the mixing part, we only show the necessary condition. But this immediately follows from a result of Bonet \cite{Bonet1}: if $T$ is a mixing operator on a Banach space $X$, then the orbit of every non-zero $x^*\in X^*$ under $T^*$ satisfies $\lim_{n\rightarrow \infty}\|T^{*n}x^*\|=\infty$. In our context, choosing $x^*=k_j$ for some $j$, we get that $\lim_{n\rightarrow \infty}\|k_j\circ B_w^n\|=\infty.$
This concludes the proof. 

\end{proof}

A large class of Banach spaces of analytic functions is covered in the above theorem. From Proposition \ref{Bergman1}, it follows that the properties (E), (P) and (C) stated in the previous theorem are satisfied by $A^p_{\phi}$. (See, also, Bonet \cite{Bonet1} and M.J. Beltr\'{a}n \cite{Maria} for similar arguments on the hypercyclicity of derivative operators in certain weighted spaces of analytic functions.)

\begin{theorem}\label{Bergman2}
For $1\leq p<\infty,$ let $A^{p}_{\phi}$ be a weighted Bergman space for a normal weight $\phi$. If $B_{w}$ is bounded on $A^{p}_{\phi},$ then the following statements are equivalent.
    \begin{itemize}
        \item[(i)] $B_{w}$ is hypercyclic on $A^{p}_{\phi}.$
        \item[(ii)] $\displaystyle \limsup_{n\rightarrow \infty}\left({\lvert w_{1}\cdots w_{n} \rvert}\Big(\int_{0}^{1} r^{pn+1} \phi(r)dr\Big)^{-1/p}\right)=\infty$.
        \item[(iii)] $\displaystyle \liminf_{n\rightarrow \infty}~\left(\frac{1}{|w_1\cdots w_n|}\|z^n\|_{A^p_{\phi}}\right)=0$.
  \end{itemize}
\end{theorem}

\begin{remark}\label{mix-normal}
In a similar way, we get a characterization of mixing  for weighted shifts $B_w$. Indeed, for mixing $\limsup$ and $\liminf$ become $\lim$ in the above theorem. Its proof is an application of the Gethner-Shapiro criterion with $n_k:=k$ for $k\geq 1$. 
\end{remark}

       A more simplified characterizations of the hypercyclicity and mixing of $B_w$ for the standard radial weight $(1+\alpha)(1-\lvert z\rvert^{2})^{\alpha},$ $-1<\alpha<\infty,$ and the logarithmic type weight $(1-|z|)^{a}(\log \frac{e}{1-|z|})^{b},~ a> b> 0,$ are given below. These follow at once, from the preceding Theorem \ref{Bergman2} and Proposition \ref{integration-pro}.
\begin{corollary}\label{Hyper}
    For $1\leq p<\infty,$ the following hold, assuming that $B_w$ is bounded on $A^p_{\alpha}$.
    \begin{itemize}
        \item[(i)] $B_w$ is hypercyclic on $A^{p}_{\alpha}$ $\Longleftrightarrow$ $\sup_{n\geq 1}~~\left( \lvert w_{1}\cdots w_{n} \rvert (n+1)^{\frac{\alpha+1}{p}} \right)=\infty.$
        \item[(ii)] $B_w$ is mixing on $A^{p}_{\alpha}$ $\Longleftrightarrow$ 
        $\lim_{n\rightarrow \infty}~~\left( {\lvert w_{1}\cdots w_{n} \rvert} (n+1)^{\frac{\alpha+1}{p}}\right)=\infty.$
    \end{itemize}
\end{corollary}

\begin{corollary}\label{Hyperc}
Let $1 \leq p < \infty$, and $B_w$ be bounded on $A^p_{a,b}$. Then the following statements hold:

\begin{enumerate}
    \item[(i)]$B_w$ is hypercyclic on $A^{p}_{a,b}$ $\Longleftrightarrow$ $\sup_{n\geq 1}~~\left( \lvert w_{1}\cdots w_{n} \rvert \frac{(n+1)^{\frac{a+1}{p}}}{(\log(n+1))^{\frac{b}{p}}} \right)=\infty.$
        \item[(ii)] $B_w$ is mixing on $A^{p}_{a,b}$ $\Longleftrightarrow$ 
        $\lim_{n\rightarrow \infty}~~\left( {\lvert w_{1}\cdots w_{n} \rvert}\frac{(n+1)^{\frac{a+1}{p}}}{(\log(n+1))^{\frac{b}{p}}}\right)=\infty.$
\end{enumerate}
\end{corollary}



We provide one more corollary to Theorem \ref{Bergman2}. This is a generalization of a result of Gethner and Shapiro, mentioned in Section $1$.

\begin{corollary}\label{normal-mixing}
    Suppose $\phi(z)$ is a normal weight. Then the unweighted shift $B$ is mixing on $A^p_{\phi}$, where $1\leq p<\infty$.
\end{corollary}

\begin{proof}
    Since $\phi$ is normal, there are constants $M>0$, $0<R<1$, and $\epsilon>-1$ such that 
$\phi(r)\leq M (1-r)^{\epsilon},$ for $r\in (R,1)$, and hence, we get
    \[
    \int_{R}^1r^{pn+1} \phi(r)dr\leq M\int_{R}^{1}r^{pn+1}(1-r)^{\epsilon}dr.
    \]
    The right side integral can be expressed in beta functions, and it will converge to $0$, as $n\rightarrow \infty$, in view of Proposition \ref{integration-pro}. Also, since $\phi$ is integrable, we see that 
    \[
    \int_{0}^Rr^{pn+1} \phi(r)dr \leq R^{pn+1}\int_0^1\phi(r)dr \rightarrow 0,
    \]
    as $n\rightarrow \infty$. Remark \ref{mix-normal} yields that $B$ is mixing.
    \end{proof}

\begin{remark}
As we mentioned earlier, Theorem \ref{bsaf} is applicable to several Banach spaces of analytic functions. We provide here, the class of the Hardy space $H^p$ consisting of all analytic functions $f(z)$ on the unit disc $\mathbb{D}$ such that $\|f\|^p:=\sup_{0<r<1}\int_0^{2\pi}|f(re^{i\theta})|^p\frac{d\theta}{2\pi}<\infty,$ $1\leq p<\infty$. Evaluation functionals are continuous, and the polynomials are dense, cf. \cite{Zhu1}. Since the inequality (C) of Theorem \ref{bsaf} holds in $H^p$, it follows that $B_w$ is hypercyclic on $H^p$ if and only if $\limsup_{n\geq 1}|w_1\cdots w_n|=\infty.$
\end{remark}

\begin{remark}\label{non-int}
   If we take the weight $\phi(r) = (1-r)^{-1}\left(\log\frac{e}{1-r}\right)^{-(1+\alpha)},$ $\alpha>0,$ then the corresponding Bergman space $A^{p,\alpha}_{\log}$ lies strictly between the Hardy space $H^{p}$ and the classical Bergman space $A^{p}$, for $1\le p<\infty$; see, for example, \cite{Zhan}, p.~3. The operator $B$ is mixing on $A^{p,\alpha}_{\log}$, which can be concluded from the equivalence
\[
\|z^{n}\|^{p}_{A^{p,\alpha}_{\log}}:= \int_{0}^{1} r^{pn+1}(1-r)^{-1}\Big(\log\frac{e}{1-r}\Big)^{-(\alpha+1)}\,dr
    \asymp (\log(n+1))^{-\alpha},\qquad n\to\infty.
\]
Note that this weight $\phi(r)$ is not integrable.
\end{remark}

 We now switch towards necessary and sufficient conditions for $B_w$ to be chaotic. Our emphasis is on the chaos of $B_w$ on the standard Bergman space $A^p_{\phi}$ for normal weights. In particular, we will show that $B$ is chaotic on $A^p_{\alpha}$ precisely when $p<2+\alpha$. We will need the following result.

\begin{proposition}[\cite{Zhu}, p. 8]\label{int} For $p>1$, as $|\lambda| \to 1^{-},$
\[
\int_{\mathbb{D}}\left|\frac{1}{1-\lambda z}\right|^{p}(1-|z|^{2})^{\alpha}\, dA(z) \asymp
\begin{cases}
  1,  & \text{if $2+\alpha>p$}, \\
 \log\frac{1}{1-|\lambda|^{2}}, & \text{if $2+\alpha=p$}, \\
   \frac{1}{(1-|\lambda|^{2})^{p-2-\alpha}}, & \text{if $2+\alpha<p$}.
\end{cases}
\]
\end{proposition}
Our results on chaos are presented in the following theorem.

\begin{theorem}\label{T}
    For $1\leq p<\infty,$ the following hold, assuming that $B_w$ is bounded.
 \begin{itemize}
     \item[(i)] If $B_{w}$ has a non-zero periodic vector in $A^p_{\phi}$ for a radial weight $\phi$, then 
     \[
     \inf_{n \geq 1} 
\Big(|w_{1}w_{2}\cdots w_{n}|\left(\!\int_{0}^{1} r^{pn+1}\phi(r) dr\!\right)^{-\frac{1}{p}}\Big)>0.
     \]
     The later condition for standard weights becomes 
     \[
     \inf_{n\geq 1}\left(|w_{1}w_{2}\cdots w_{n}|(n+1)^{\frac{\alpha+1}{p}}\right)>0.
     \]
     \item[(ii)] If $1<p<2+\alpha$ and 
     \[
     \sum_{n=1}^{\infty}\frac{1}{|w_{1}w_{2}\cdots w_{n}|n^{\frac{2+\alpha-p}{p}}}<\infty,
     \] 
     then $B_{w}$ is chaotic on $A^{p}_{\alpha}.$
     \item[(iii)] The unweighted shift $B$ is chaotic on $A^p_{\alpha}$ if and only if $p<2+\alpha$.
 \end{itemize}  
\end{theorem}
\begin{proof}
 (i) Assume that $f(z)=\sum_{n=0}^{\infty}\lambda_{n}z^{n}$ be a non-zero periodic vector for $B_{w}$ on $A^{p}_{\phi}.$ Then for some $m\in \mathbb{N},$ we have $B_{w}^{m}f(z)=f(z),$ for all $z\in \mathbb{D}.$ Since $B_{w}^{km}f(z)=f(z)$ for all $k\geq 1$, it follows that 
 \begin{eqnarray*}
\sum_{n=km}^{\infty} w_{n}w_{n-1}\cdots w_{n-km+1}\lambda_{n}z^{n-km}=\sum_{n=0}^{\infty} \lambda_{n}z^{n},
\end{eqnarray*}
for all $z\in \mathbb{D}$. As $m$ be the period of $B_{w}$, then there is some $j\leq m$ such that $\lambda_{j}\neq 0.$ Using the coefficients of like powers for comparison, we obtain 
$w_{t+1}w_{t+2}\cdots w_{t+km}\lambda_{t+km}=\lambda_{t}$, $\forall~~ k \geq 1,$ where $ t \geq j$. This implies that
\begin{equation}\label{new1}
    \frac{\lambda_{t}}{w_{t+1}w_{t+2}\cdots w_{t+km}}=\lambda_{t+km}.
\end{equation}
Using the Cauchy integral formula for $f \in A^{p}_{\phi}$, there exists a constant $C>0$ such that 
\[
|\lambda_{t+km}| = \frac{|f^{(t+km)}(0)|}{(t+km)!}
\leq \frac{C}{r^{t+km}} \left( \int_{-\pi}^{\pi} |f(re^{i\theta})|^{p} d\theta \right)^{1/p},
\quad 0 < r < 1.
\]
Multiplying both sides by $r^{p(t+km)+1}\phi(r)$ and integrating over $[0,1]$ yields
\[
\left(\!\int_{0}^{1} r^{p(t+km)+1}\phi(r) dr\!\right)
\left(\frac{|f^{(t+km)}(0)|}{(t+km)!}\right)^{p}
\leq C \|f\|_{A^{p}_{\phi}}^{p}.
\]
Therefore
\begin{equation}\label{new2}
|\lambda_{t+km}| \leq \frac{C_{p}}{\left(\!\int_{0}^{1} r^{p(t+km)+1}\phi(r) dr\!\right)^{\frac{1}{p}}}\|f\|_{A^{p}_{\phi}}.
\end{equation}
Combining (\ref{new1}) and (\ref{new2}) gives
\[
\left| \frac{\lambda_{t}}{w_{t+1}w_{t+2}\cdots w_{t+km}} \right|
\leq \frac{C_{p}}{\left(\!\int_{0}^{1} r^{p(t+km)+1}\phi(r) dr\!\right)^{\frac{1}{p}}}\|f\|_{A^{p}_{\phi}},
\quad t \geq j,\; k \geq 1.
\]
Consequently,
\[
\sup_{n \geq j} 
\frac{\left(\!\int_{0}^{1} r^{pn+1}\phi(r) dr\!\right)^{\frac{1}{p}}}{|w_{j}w_{j+1}\cdots w_{n}|}
< \infty.
\]
This completes the proof of the first part of \textnormal{(i)}. 
For the case of the standard weight, we have 
$\int_{0}^{1} r^{pn+1}\phi(r)\,dr 
    \sim \frac{1}{(n+1)^{\alpha+1}},$ which establishes the assertion in \textnormal{(i)}.

(ii) To show that $B_{w}$ is chaotic on $A^{p}_{\alpha},$ we use the chaoticity criterion. As in the previous theorem, consider $X_0$ be the space of all polynomials and $S: X_{0}\to X_{0}$ which is given by
      \[
     S(z^{j})= \frac{z^{j+1}}{w_{j+1}},~\hspace{.0cm} j\geq 0. 
     \]
Observing that  $B_{w}S(f)=f$ and the series $\sum_{n=0}^{\infty} (B_{w})^{n}(f)$ converges unconditionally for each $f \in X_0.$ It remains to show that the series  $\sum_{n=0}^{\infty} S^{n}(f)$ converges unconditionally, for each $f \in X_0.$ We prove that
     $\sum_{n=1}^{\infty} \frac{1}{w_{1}\cdots w_{n}} z^n $
     is unconditionally convergent in $A^{p}_{\alpha}$. By the well known Orlicz-Pettis theorem (cf. Diestel \cite{Diestel}) it suffices to show that
        \[
        \sum_{n=1}^{\infty}\left|\frac{1}{w_{1}\cdots w_{n}}L(z^{n})\right|<\infty,
        \]
        for every bounded linear functional $L$  on $A^{p}_{\alpha}.$ Since the dual of $A^p_{\alpha}$ can be identified with $A^q_{\alpha}$ under the pairing
\[
\langle f, h \rangle := (\alpha+1)\int_{\mathbb{D}} f(z) \, h(z) \, (1-|z|^2)^\alpha \, dA(z), 
\quad f \in A^p_\alpha, \ h \in A^q_\alpha, 
\]
$1/p+1/q=1$, (cf. \cite{Zhu}, p. 18), we will prove that the series  
\[
\sum_{n\geq 1} \frac{1}{w_1\cdots w_n}\int_{\mathbb{D}} z^{n}h(z)(1-|z|^{2})^{\alpha}\, dA(z),
\] 
is absolutely convergent. Now, recalling the point wise estimate as in Proposition~\ref{Berg}, we have, for $p<2+\alpha$,  
\begin{eqnarray*}
\int_{\mathbb{D}} |z|^{n}|h(z)|(1-|z|^{2})^{\alpha}\, dA(z)&\leq& C_{p}\|h\|_{A^{q}_{\alpha}}\int_{0}^{1} r^{n+1}(1-r^{2})^{\alpha-\tfrac{2+\alpha}{q}}\,dr\\[6pt]
&=& C_{p}\|h\|_{A^{q}_{\alpha}}  \frac{\Gamma\!\left(\tfrac{n}{2}+1\right)\Gamma\!\left(\tfrac{2+\alpha-p}{p}\right)}{\Gamma\!\left(\tfrac{n}{2}+\tfrac{2+\alpha}{p}\right)}.
\end{eqnarray*}
Hence, by recalling Lemma~\ref{gamma}, we can deduce that  
\[
\big|L(z^{n})\big| \leq C_{p}\|h\|_{A^{q}_{\alpha}} \, \frac{1}{n^{\tfrac{2+\alpha-p}{p}}}.
\]
Therefore, under the given assumption, we may conclude that $B_{w}$ satisfies the chaoticity criterion, and hence (ii) follows.\\
(iii) First, note that by Corollary~\ref{normal-mixing}, the backward shift operator $B$ is mixing on $A^{p}_{\alpha}$. We now show that $B$ has a dense set of periodic vectors on $A^{p}_{\alpha}$ if and only if $p<2+\alpha$. For $\lambda \in \overline{\mathbb{D}}$ and $z \in \mathbb{D}$,
$
f_{\lambda}(z):=\sum_{n=0}^{\infty}\lambda^{n}z^{n},
$
is an eigenvector of $B$ corresponding to the eigenvalue $\lambda$, provided that $f_{\lambda}\in A^p_{\alpha}$. By Proposition \ref{int}, we have $f_{\lambda}\in A^{p}_{\alpha}\quad \text{if and only if } p< 2+\alpha$ whenever $|\lambda|=1$.  In particular, if $\lambda$ is a root of unity, then $f_{\lambda}$ is a non-trivial periodic vector of $B$. To show that the periodic vectors of $B$ are dense in $A^p_{\alpha}$, it is enough to show that the family
$\mathcal{F}:=\Big\{f_{\lambda}:~ \lambda \text{ a root of unity}\Big\}$ spans a dense subspace in $A^{p}_{\alpha}$. To this end, let $h \in A^{q}_{\alpha}$, with $\tfrac{1}{p}+\tfrac{1}{q}=1$, and
\begin{equation}\label{i}
\int_{\mathbb{D}} f_{\lambda}(z)\, h(z)\,(1-|z|^2)^\alpha\, dA(z) = 0.
\end{equation}
 Expanding $f_{\lambda}(z)$ as a power series, we get that $\sum_{n=0}^{\infty} \lambda^{n} M_{n} = 0,$ where
 \[
M_n := \int_{\mathbb{D}} z^n h(z)\,(1-|z|^2)^\alpha\, dA(z).
\]
 Note that $\sum_{n=0}^{\infty} M_ne^{in\theta}$ converges for all $\theta \in [-\pi,\pi)$, and it defines a continuous function. Also, since $\sum_{n=0}^{\infty} M_ne^{in\theta}$ vanishes when $\theta$ is a rational multiple of $\pi$, we get that 
 \[
 \sum_{n=0}^{\infty} M_{n} e^{in\theta}=0,~\forall~\theta\in[-\pi,\pi).
 \]
 By the uniqueness of Fourier coefficients, we must have $M_{n}=0$ for all $n\geq 0$. Since the polynomials are dense in $A^{p}_{\alpha}$, it follows that $h\equiv 0$. Consequently, the span of $\mathcal{F}$ is dense in $A^{p}_{\alpha}$, completing the proof of (iii).
\end{proof}

\subsection{Shifts on non-radial Bergman spaces}

We provide examples of non-radially weighted Bergman spaces $A^p_{\phi}$, and study the dynamics of $B_w$. The weight functions considered in this subsection are the non-radial weights of Section $2$. In the Hilbert space case of $A^2_{\phi},$ it is evident from the proofs of (ii) and (iii) in Theorem \ref{T} that we can have the following general results, since, in this case the dual space is known, and also, the polynomials are dense. 
\begin{proposition}
Let $\phi(z):=|g(z)|^{\gamma}$, where $g(z)$ and $\gamma$ are as in Corollary \ref{Herg}. If
\[
\sum_{n\geq 1}\frac{1}{|w_{1}\cdots w_{n}|}\int_{0}^{1}\int_{0}^{2\pi}r^{n+1}(1-r)^{-1}|g(re^{i\theta})|^{\frac{\gamma}{2}}drd\theta<\infty,
\]
 then $B_{w}$ is chaotic on $A^{2}_{\phi}.$
\end{proposition}
\begin{proof}
    We only need to show that, for every $h \in A^2_{\phi}$, the series  
\[
\sum_{n\geq 1} \frac{1}{w_1\cdots w_n}\int_{\mathbb{D}} z^{n}\overline{h(z)}\,|g(z)|^{\gamma}\, dA(z)
\]
is absolutely convergent. Recalling the point wise estimate from Proposition~\ref{Noon-radial2}, we obtain  
\[
\int_{\mathbb{D}} |z|^{n}\,|h(z)|\,|g(z)|^{\gamma}\, dA(z)
\leq
\int_{0}^{1}\int_{0}^{2\pi}
r^{\,n+1} (1-r)^{-1}
\,|g(re^{i\theta})|^{\frac{\gamma}{2}}
\, d\theta\, dr.
\]
Using this estimate together with our assumption, we conclude that $B_w$ is chaotic on $A^2_{\phi}$.
\end{proof}
\begin{proposition}
    Let $\phi(z)$ be a weight (not necessarily radial), and let $\frac{1}{1-\lambda z}\in A^2_{\phi}$, for all unimodular $\lambda$. Then, $B$ is chaotic and mixing.
\end{proposition}
\begin{proof}
    It is essentially the same as the proof of Theorem \ref{T} (ii).
\end{proof}

 \begin{proposition}
    For $1\leq p<\infty,$ let $A^{p}_{\phi}$ be a weighted Bergman space corresponding to the non-radial weight $\phi(z)=1-|\Re(z)|.$ Then the following hold.
    \begin{itemize}
        \item[(i)] The unweighted shift $B$ is bounded and mixing on $A^{p}_{\phi}.$
        \item[(ii)] Assume that the weighted backward shift $B_{w}$ is bounded on $A^{p}_{\phi}$. Then it is hypercyclic on $A^{p}_{\phi}$ if 
$\displaystyle \sup_{n\geq 1}\big( \lvert w_{1}\cdots w_{n} \rvert (n+1)^{\frac{1}{p}} \big)=\infty$. 
Conversely, if $B_{w}$ is hypercyclic on $A^{p}_{\phi}$, then 
$\displaystyle \sup_{n\geq 1}\big( \lvert w_{1}\cdots w_{n} \rvert (n+1)^{\frac{3}{p}} \big)=\infty$. 
    \end{itemize}
\end{proposition}
\begin{proof}
{\text{(i)}} By Theorem \ref{bounded1}, $B$ is bounded on $A^{p}_{\phi}$. Further, applying the Gethner-Shairo criterion, we see that $B$ is mixing as $\lim_{n\rightarrow \infty}\|z^n\|=0$. Indeed,
\[
          \int_{\mathbb{D}} |z^{n}|^{p}\phi(z)dA(z)=\int_{\theta=0}^{2\pi}\int_{r=0}^{1} r^{pn+1}(1-r|\cos\theta|) \frac{drd\theta}{\pi}=\frac{2}{pn+2}-\frac{4}{\pi(pn+3)}.
\]
        
(ii) Applying the Gethner-Shapiro criterion to the dense subspace $X_0$ of polynomials, and the obvious right inverse of $B_w$ on $X_0$, we can see that $B_{w}$ is hypercyclic on $A^{p}_{\phi}$ if\\
$\inf_{n}\left({\lvert w_{1}\cdots w_{n} \rvert^{-p}}\int_{\mathbb{D}} |z^{n}|^{p}\phi(z)dA(z)\right)=0.
$
The later condition follows from the sufficient part in (ii). Conversely, assume that $B_w$ is hypercyclic. If $f(z)$ is a hypercyclic vector for $B_w$, then \\ $\sup_{n\geq 0}~\left\lvert {w_{1}w_{2}\cdots w_{n}} \frac{f^{(n)}(0)}{n!}\right\rvert=\infty.$ Now, using the coefficient estimate from Proposition \ref{non-radial}, we get 
 $ \sup_{n\geq 1}~~\Big({\lvert w_{1}\cdots w_{n} \rvert(n+1)^{\frac{3}{p}}} \Big)=\infty.
 $ This gives the converse part in \text{(ii)}.
\end{proof}
Similarly, the hypercyclicity of $B_{w}$ corresponding to the weight of the form $|g(z)|^{\gamma}$ can be established.
\begin{proposition}
    For $1\leq p<\infty,$ let $A^{p}_{\phi}$ be a weighted Bergman space corresponding to the non-radial weight $\phi(z)=|g(z)|^{\gamma},$ where $0<\gamma<1$, and $g$ is as in Corollary \ref{Herg}. Assume that the weighted backward shift $B_{w}$ is bounded on $A^{p}_{\phi}$. Then the following hold.
    \begin{itemize}
        \item[(i)] $B_{w}$ is hypercyclic on $A^{p}_{\phi}$ if 
$\displaystyle \sup_{n\geq 1}\left(| w_{1}\cdots w_{n}|\big( \int_{0}^{1}\int_{0}^{2\pi}r^{pn+1} |g(re^{i\theta})|^{\gamma} drd\theta\big)^{-\frac{1}{p}}\right)=\infty$.
        \item[(ii)]   
Conversely, if $B_{w}$ is hypercyclic on $A^{p}_{\phi}$, then 
\[
\displaystyle \sup_{n\geq 1}\left( \lvert w_{1}\cdots w_{n} \rvert \frac{1}{r^{n}(1-r)^{\frac{2}{p}}}\int_{0}^{2\pi}\frac{d\theta}{|g(re^{i\theta})|^{\frac{\gamma}{p}}} \right)=\infty,\quad 0<r<1.
\]
    \end{itemize}
\end{proposition}

It would be interesting to obtain a characterization of hypercyclic weighted shifts on $A^p_{\phi}$ for some classes of non-radial weights.
\vskip .3cm
\noindent\textbf{Acknowledgments:} B.K. Das is supported by a CSIR research fellowship \\(File No.: 09/1059(0037)/2020-EMR-I), and A. Mundayadan is partially funded by a Start-Up Research Grant of SERB-DST (File. No.: SRG/2021/002418).
	
	\noindent\textbf{Competing interests declaration:} The authors declare none. 
	
\bibliographystyle{amsplain}

\end{document}